\documentclass[11pt]{article}
\usepackage{mathrsfs}
\usepackage{amssymb}
\usepackage{amsmath,amsthm,amssymb,mathrsfs,amsbsy,bm}
\usepackage{graphicx}
\usepackage{epsfig}
\usepackage{enumerate}
\usepackage{color, xcolor} %%

\usepackage[colorlinks=true, allcolors=blue]{hyperref}
\usepackage[letterpaper,top=2cm,bottom=2cm,left=3cm,right=3cm,marginparwidth=1.75cm]{geometry}

\numberwithin{equation}{section}

\newtheorem{theorem}{\bf{Theorem}}[section]
\newtheorem{lemma}{\bf {Lemma}}[section]
\newtheorem{define}{\bf{Definition}}[section]
\newtheorem{corollary}{\bf{Corollary}}[section]

\newtheorem{remark}{\bf{Remark}}[section]

\newcommand{\mbE}{\widehat{\mathbb{E}}}
\newcommand{\mbe}{\widehat{\mathcal{E}}}
\newcommand{\V}{\mathbb{V}}
\newcommand{\mv}{\mathcal{V}}
\newcommand{\sles}{(\Omega,\mathcal{H},\mbE)}

\newcommand*{\dif}{\mathop{}\!\mathrm{d}}

\newcommand{\bE}{\breve{\mathbb{E}}}
\newcommand{\be}{\breve{\mathcal{E}}}

\title{Strong laws of large numbers for sequences of blockwise $m$-dependent and orthogonal random variables under sublinear expectations}
\author{Jialiang Fu\\ \sl \small Academy of Mathematics and Systems Science, Chinese Academy of Sciences, \\ \sl \small Beijing, China\\ \sl \small E-mail: fujialiang@amss.ac.cn }
\date{}

\begin{document}

\maketitle

\begin{abstract}
\par In this paper, we establish some strong laws of large numbers (SLLN) for non-independent random variables under the framework of sublinear expectations. One of our main results is for blockwise  $m$-dependent random variables, and another is for orthogonal random variables. Both are the generalizations of SLLN for independent random variables in sublinear expectation spaces.   
\end{abstract}

\section{Introduction}
 \ \ \ \ \  Kolmogorov's strong law of large numbers is a footstone in classical probability theory, which means the sample mean converges to the population mean. This fundamental theorem is based on the assumption that the random variables are independent and identically distributed. However, due to the complexity of the real world, data can also be dependent and not identically distributed in the observation of data. So it is very necessary to extend the law of large numbers to the case of dependent random variables. 
 
 A kind of common dependence is $m$-dependence, which was proposed in Hoeffding and Robbins's paper \cite{HR48}  in 1948 to our knowledge. A sequence $\{X_n\}_{n\geq1}$ of random variables is called $m$-dependent if $\{X_1,...,X_r\}$ is independent of $\{X_s,X_{s+1},...\}$ provided $s-r>m$, where $m$ is a fixed nonnegative integer. It's clear that 0-dependence is independence. Moricz \cite{Mo87} introduced the definition of blockwise $m$-dependence and obtained a strong law of large numbers for the dyadic block $\{X_k:2^{p-1}<k\leq2^p\}, p\in \mathbb{N}^*$. Later, Gaposhkin \cite{Gap95} and Zhang \cite{Bo98} generalized the dyadic block to arbitrary blocks. In theoretical analysis and practice, it may be much easier to check the orthogonality than independence. So it is reasonable to consider the law of large numbers for orthogonal random variables. Moricz \cite{Mo87} also obtained a strong law of large numbers for blockwise quasi-orthogonal sequences of random variables.
 
 Sublinear expectations are introduced by Shige Peng, under whose framework the expectations are not linearly additive anymore. Under the framework of nonlinear expectations, limit theorems and stochastic analysis can also be established corresponding to classical linear expectations. In recent years, there have been so many papers about the strong law of large numbers under sublinear expectations. Most of them are based on the assumption of independence, but the results for dependent random variables are relatively few. In recent work, Zhang \cite{GZ24} obtained a strong law of large numbers for $m$-dependent and stationary random variables under sublinear expectations. In this paper, we shall establish strong laws of large numbers for blockwise $m$-dependent random variables and orthogonal random variables respectively.
 
 The structure of this paper is as follows. In section 2, we shall present some preliminaries for the sublinear expectations. In section 3, some inequalities and properties needed in our proof will be presented. We shall present our main theorems in section 4 and the detailed proofs will be presented in section 5.

\section{Preliminaries}
\ \ \ \ \ In this section, we shall present some basic notations and results of sublinear expectations under the framework of Peng and one can refer to \cite{Pen19} for more details. 
\par Let $(\Omega,\mathcal{F})$ be a given measurable space and let $\mathcal{H}$ be a linear space of real functions defined on  $(\Omega,\mathcal{F})$. As the space of the random variables, $\mathcal{H}$ satifies that if $X_1,\cdots, X_n\in\mathcal{H}$, then $\varphi(X_1,\cdots,X_n)\in\mathcal{H}$ for each $\varphi\in C_{l,Lip}(\mathbb{R}^n)$. $C_{l,Lip}(\mathbb{R}^n)$ denotes the linear space of local Lipschitz functions $\varphi$ satisfying
\begin{align*}
\vert \varphi(\bm{x})-\varphi(\bm{y})& \vert\leq C(1+\vert \bm{x}\vert^m+\vert\bm{y}\vert^m)\vert\bm{x}-\bm{y}\vert,\quad\forall\bm{x},\bm{y}\in\mathbb{R}^n,\\
 & \text{ for some } C>0, m\in\mathbb{N} \text{  depending on } \varphi.
 \end{align*}
 And $C_{b, Lip}(\mathbb{R}^n)$ denotes the space of bounded Lipschitz functions.
\begin{define}
A sublinear expectation $\mbE$ on $\mathcal{H}$ is a functional $\mbE:\mathcal{H}\rightarrow{\mathbb{R}}$ satisfying:
\begin{itemize}
\item[(i)] Monotonicity: $\mbE[X]\leq\mbE[Y]$ if $X\leq Y$;
\item[(ii)] Constant preserving: $\mbE[c]=c$ for $c\in$R;
\item[(iii)] Sub-additivity: for each $X,Y\in\mathcal{H}$, $\mbE[X+Y]\leq\mbE[X]+\mbE[Y]$ ;
\item[(iv)] Positive homogeneity: $\mbE[\lambda X]=\lambda\mbE[X]$ for $\lambda\geq0$.
\end{itemize}
 The triple $(\Omega,\mathcal{H},\mbE)$ is called a sublinear expectation space. Given a sublinear expectation $\mbE$, $\mbe$ of $\mbE$ is defined by
$$
\mbe[X]\triangleq-\mbE[-X],\quad\forall X\in\mathcal{H}.
$$
\end{define}
\par By the sub-additivity of $\mbE$, it can be checked that $\mbE[X-Y]\geq\mbE[X]-\mbE[Y]$ for all $X,Y\in\mathcal{H}$, $\mbe[X]\leq\mbE[X]$, $\enspace \mbE[X+c]=\mbE[X]+c$ for $c\in R$. The last one is called cash translatability. $\mbE[X]$ and $\mbe[X]$ are often called the upper-expectation and lower-expectation of $X$ respectively.

\begin{define}
Let $(\Omega_1,\mathcal{H}_1,\mbE_1)$ and $(\Omega_2,\mathcal{H}_2,\mbE_2)$ be two sublinear expectation spaces. And a n-dimensional random vector $\bm{X_1}$ in $(\Omega_1,\mathcal{H}_1,\mbE_1)$ is said to be identically distributed with another n-dimensional random vector $\bm{X_2}$ in $(\Omega_2,\mathcal{H}_2,\mbE_2)$, denoted by $\bm{X_1}\overset{d}{=}\bm{X_2}$, if
$$
	\mbE_1[\varphi(\bm{X}_1)]=\mbE_2[\varphi(\bm{X}_2)],\quad\forall\varphi\in C_{b,Lip}(\mathbb{R}^n).
$$
 A sequence $\{X_n;n\geq 1\}$ of random variables is said to be identically distributed if $X_i\overset{d}{=}X_1$ for each $i\geq 1$.

\end{define}
\begin{define}
	Let $(\Omega,\mathcal{H},\mbE)$ be a sublinear expectation space. A random vector $\bm{Y}=(Y_1,\cdots,Y_n)\in\mathcal{H}^n$ is said to be independent of another random vector $\bm{X}=(X_1,\cdots,X_m)\in\mathcal{H}^m$ under $\mbE$ if for each test function $\varphi\in C_{b,Lip}(\mathbb{R}^{m+n})$ we have $$\mbE[\varphi(\bm{X},\bm{Y})]=\mbE[\mbE[\varphi(\bm{x},\bm{Y})]\vert_{\bm{x}=\bm{X}}].$$
\end{define}
 It is important to observe that under the framework of sublinear expectation, $Y$ is independent of $X$ does not in general imply that $X$ is independent of $Y$, which is different from the classical linear expectation. One can check the Example 1.3.15 in \cite{Pen19} for details. A sequence of random variables $\{X_n;n\geq 1\}$ is said to be independent if $X_{i+1}$ is independent of $(X_1,\cdots,X_i)$ for each $i\geq 1$. It is easy to check that if $\{X_1,\cdots,X_n\}$ are independent, then $\mbE[\sum_{i=1}^nX_i]=\sum_{i=1}^n\mbE[X_i]$.
 \begin{define}
 	A sequence of random variables $\{X_n;n\geq 1\}$ in $(\Omega,\mathcal{H},\mbE) $ is said to be $m$-dependent if there exists an integer $m$ such that $(X_{n+m+1},\cdots,X_{n+j})$ is independent of $(X_1,\cdots,X_n)$ for every $n$ and every $j\geq m+1$. In particular,  $m=0$ means $\{X_n;n\geq 1\}$ is an independent sequence.
 \end{define}
 \begin{define}
    Let $\{n_i;i\geq1\}$ be a given strictly increasing sequence of natural numbers. A sequence of random variables $\{X_n;n\geq 1\}$ in $(\Omega,\mathcal{H},\mbE) $ is said to be blockwise $m$-dependent with respect to $\{n_i;i\geq1\}$ if $\{X_{n_i\leq n<n_{i+1}}\}$ is m-dependent, which means either $n_{i+1}-n_{i}\leq m+1$ or $n_{i+1}-n_{i}> m+1$ and $\{X_n;s\leq n<n_{i+1}\}$ is independent of $\{X_n;n_i\leq n \leq r\}$ if only $s-r>m$.
 \end{define}
 \begin{define}
 	A sequence of random variables $\{X_n;n\geq 1\}$ in $(\Omega,\mathcal{H},\mbE) $ is said to be orthogonal if $\mbE[X_iX_j]=0$ for $i\neq j$. It's said to be orthonormal if $\mbE[X_iX_j]=\delta_{ij}$ and $\delta_{ij}$ here is the Kronecker symbol. 
 \end{define}
 \begin{define}\label{de1}
 	A sequence of random variables $\{X_n;n\geq 1\}$ in $(\Omega,\mathcal{H},\mbE) $ is said to be quasi-orthogonal if there exists a nonnegative sequence $\{f(j):j=0,1,\cdots\}$ and $\sum_{j=0}^{\infty}f(j)<\infty$ such that $|\mbE[X_kX_l]|\leq\sqrt{\mbE[X^2_k]}\sqrt{\mbE[X^2_l]}\cdot f(|k-l|)$,$\forall k,l=1,2,\cdots$. In particular, if $f(0)=1$ and $f$ is zero on other values, then $\{X_n;n\geq 1\}$ is orthogonal.
 \end{define}
\par Next, we consider the capacities corresponding to the sublinear expectations. One can refer to \cite{Zha16a} for more details. 
\par Let $\mathcal{G}\subset\mathcal{F}$. A function $V:\mathcal{G}\rightarrow[0,1]$ is called a capacity if
$$
V(\emptyset)=0,\enspace V(\Omega)=1 \enspace and\enspace  V(A)\leq V(B)\enspace \forall A\subset B, A,B\in\mathcal{G}.
$$
It is called sub-additive if $V(A\cup B)\leq V(A)+V(B)$ for all $A,B\in\mathcal{G}$ with $A\cup B\in\mathcal{G}$.
\par Let $\sles$ be a sub-linear expectation space. We define a pair of capacities $(\hat{\V},\hat{\mv})$ as follows.
\begin{equation}
	\hat{\V}(A)\triangleq\inf\{\mbE[\xi]:I_A\leq\xi,\xi\in\mathcal{H}\}, \hat{\mv}(A)\triangleq1-\hat{\V}(A^c), \enspace\forall A\in\mathcal{F}.
\end{equation}
\par We call $\hat{\V}$ and $\hat{\mv}$ the upper and lower capacity respectively. The capacity $\hat{\V}$ has the property that
\begin{equation}
	\mbE[f]\leq\hat{\V}(A)\leq\mbE[g]\quad if\enspace f\leq I_A\leq g,f,g\in\mathcal{H}\enspace and\enspace A\in\mathcal{F},\label{eq1}
\end{equation}
and the second inequality plays a similar role to Markov inequality in classical linear expectation.
\par Next, we define the Choquet integrals $(C_{\hat{\V}},C_{\hat{\mv}})$ by
\begin{equation}
	C_V[X]\triangleq\int_0^{\infty}V(X\geq t)\dif t+\int_{-\infty}^0[V(X\geq t)-1]\dif t
\end{equation}with $V$ being replaced by $\hat{\V}$ and $\hat{\mv}$ respectively.
If $\V_1$ on the sublinear expectation space $(\Omega_1,\mathcal{H}_1,\mbE_1)$ and $\V_2$ on the sublinear expectation space $(\Omega_2,\mathcal{H}_2,\mbE_2)$ are two capacities that have the property \eqref{eq1}, then for any random variables $X_1\in\mathcal{H}_1$ and $X_2\in\mathcal{H}_2$ with $X_1\overset{d}{=}X_2$, we have
\begin{equation}
	\V_1(X_1\geq x+\epsilon)\leq \V_2(X_2\geq x)\leq\V_1(X_1\geq x-\epsilon)\quad for\enspace all\enspace\epsilon>0\enspace and \enspace x,
\end{equation}
and
\begin{equation}
	C_{\V_1}[X_1]=C_{\V_2}[X_2].
\end{equation}
\par Since $\hat{\V}$ may be not countably sub-additive so that the Borel-Cantelli lemma is not valid, we consider the outer capacity $\hat{\V}^*$ which defined in \cite{Zha16b} by
\begin{equation}
	\hat{\V}^*(A)\triangleq\inf\left\{\sum_{n=1}^{\infty}\hat{\V}(A_n):A\subset\bigcup_{n=1}^{\infty}A_n, A_n\in \mathcal{F}, n\geq1\right\},\hat{\mv}^*(A)\triangleq1-\hat{\V}^*(A^c),\quad A\in\mathcal{F}.
\end{equation}
As shown in Zhang \cite{Zha16b}, $\hat{\V}^*$ is countably sub-additive, $\hat{\V}^*(A)\leq\hat{\V}(A)$ and satisfies $\hat{\V}^*(A)\leq\mbE[g]$ whenenver $I_A\leq g\in\mathcal{H}$. Further, $\hat{\V}(A)$(resp. $\hat{\V}^*$) is the largest sub-additive(resp. countably sub-additive) capacity in sense that if $V$ is also a sub-additive(resp. countably sub-additive) capacity satisfying $V(A)\leq\mbE[g]$ whenenver $I_A\leq g\in\mathcal{H}$, then $V(A)\leq\hat{\V}(A)$(resp. $V(A)\leq\hat{\V}^*(A)$).
\par  Throughout the whole paper, we denote $x\vee y\triangleq\max\{x,y\}, x\wedge y\triangleq\min\{x,y\}, x^+\triangleq x\vee0, x^-\triangleq(-x)\vee0$ for real numbers $x$ and $y$. $[x]$ means the maximum integer not exceeding $x$. $\mathbb{N}$ represents all natural numbers and $\mathbb{N}^*$ represents all non-zero natural numbers. $0<C$ is a constant that may change from line to line. For a random variable $X\in\mathcal{H}$, we truncate it in the form $(-c)\vee X\wedge c$ denoted by $X^{(c)}$ because $XI\{\vert X\vert\leq c\}$ may not be in $\mathcal{H}$. We define $\bE[X]\triangleq\lim_{c\rightarrow\infty}\mbE[X^{(c)}]$ if the limit exists, and $\be[X]\triangleq-\bE[-X]$. It's clear that $\bE[X]=\mbE[X]$ if $X$ is bounded.

\section{Some lemmas and inequalities}
\ \ \ \ \ In this section, we shall present some important lemmas and the crucial inequalities in our proofs of main results. 
\begin{lemma}\label{le3.1}
Suppose $X\in\mathcal{H}$,
\begin{itemize}
	\item[(i)] For any $0\leq c<\infty$,$$\mbE[|X|\wedge c]\leq \int_{0}^{c}\hat{\V}(|X|>x)dx.$$
	\par If \ $\lim_{c\rightarrow\infty}\mbE[(|X|-c)^+]=0$, then 
	\begin{equation}
		\mbE[|X|]\leq C_{\hat{\V}}[|X|].\label{eq3}
	\end{equation}
	\item[(ii)] If $C_{\hat{\V}}[|X|]<\infty$, then $\bE[|X|]$ is well defined and $\bE[|X|]\leq C_{\hat{\V}}[|X|]$. $\bE[X]$ is a sublinear expectation on $\mathcal{H}_1\triangleq\{X\in\mathcal{H}|\ \lim_{c,d\rightarrow\infty}\mbE[(|X|\wedge d-c)^+]=0\}\supseteq\{X\in\mathcal{H}|\ C_{\hat{\V}}[|X|]<\infty\}$. 
	\item[(iii)] If $C_{\hat{\V}}[|X|^r]<\infty$ for some $r>1$, then $C_{\hat{\V}}[|X|]<\infty$.
\end{itemize}
\end{lemma}

\begin{proof}
	The proof of $(i)$ and $(ii)$ can be found in \cite{Zha16a} and \cite{GZ24}, and for $(iii)$ we notice 
\begin{align}
	C_{\hat{\V}}[|X|^r]&=\int_{0}^{+\infty}\hat{\V}(|X|^r\geq t)dt\notag\\
	&=\int_{0}^{+\infty}\hat{\V}(|X|\geq t^{\frac{1}{r}})dt\notag\\
	&=\int_{0}^{+\infty}rm^{r-1}\hat{\V}(|X|\geq m)dm\notag\\
	&\geq \int_{1}^{+\infty}\hat{\V}(|X|\geq m)dm.\notag
\end{align}
	So if $C_{\hat{\V}}[|X|^r]<\infty$, then $C_{\hat{\V}}[|X|]<\infty$.
\end{proof} 
The next lemma is about exponential inequalities and Kolmogorov’s maximal inequalities in Zhang \cite{GZ24} and we shall use the latter
to obtain a maximal inequality for $m$-dependent random variables under sublinear expectations.

\begin{lemma}
 Let $\{X_1,\cdots,X_n\}$ be a sequence of independent random variables in the sublinear expectation space $(\Omega,\mathcal{H},\mbE)$. Set $S_n=\sum_{i=1}^n X_i, B_n^2=\sum_{i=1}^n \mbE[X_i^2]$. Then for all $x>0$, $0<\delta\leq1$ and $p\ge 2$,
\begin{align}
	&\hat{\V}\left(\max_{k\leq n}(S_k-\mbE[S_k])\geq x\right)\quad\left(resp.\hat{\mv}\left(\max_{k\leq n}(S_k-\mbe[S_k])\geq x\right)\right)\nonumber\\
	\leq&C_p\delta^{-p}x^{-p}\sum_{i=1}^n\mbE[X_i^2]+\exp\left\{-\frac{x^2}{2(1+\delta)B_n^2}\right\}.\nonumber
\end{align}
Further, by noting that $xe^{-x}\leq e^{-1}$ when $x\geq0$, we have Kolmogorov’s maximal inequalities as follows:
\begin{align}
	&\hat{\V}\left(\max_{k\leq n}(S_k-\mbE[S_k])\geq x\right)\leq Cx^{-2}\sum_{i=1}^n\mbE[X_i^2],\label{eq2}\\
	&\hat{\mv}\left(\max_{k\leq n}(S_k-\mbe[S_k])\geq x\right)\leq Cx^{-2}\sum_{i=1}^n\mbE[X_i^2].\nonumber
\end{align}
\end{lemma}
By \eqref{eq2}, we can obtain a maximal inequality for $m$-dependent random variables, which will play a crucial role in our proof of the strong law of large numbers for blockwise $m$-dependent random variables.
\begin{lemma}\label{le3.3}
	Let $\{X_1,\cdots,X_n\}$ be a sequence of m-dependent random variables in the sublinear expectation space $(\Omega,\mathcal{H},\mbE) $. Then for all $x>0$,
	\begin{equation}
		\hat{\V}\left(\max_{k\leq n}\sum_{i=1}^{k}(X_i-\mbE[X_i])\geq x\right)\leq Cx^{-2}\sum_{i=1}^{n}\mbE[X_i^2],
	\end{equation}
	where C is a constant independent of $n$.
\end{lemma}
\begin{proof}
In case of $n< m+1$, we take $Z_i\triangleq X_i-\mbE[X_i]$. Then
\begin{align}
	\hat{\V}\left(\max_{k\leq n}\sum_{i=1}^{k}Z_i\geq x\right)
	&=\hat{\V}\left(\bigcup_{k=1}^{n}\left\{\sum_{i=1}^{k}Z_i\geq x\right\}\right)\nonumber\\
	&\leq\sum_{k=1}^{n}\hat{\V}\left(\sum_{i=1}^{k}Z_i\geq x\right)\nonumber\\
	&\leq\sum_{k=1}^{n}\hat{\V}\left(\bigcup_{i=1}^{k}\left\{Z_i\geq \frac{x}{k}\right\}\right)\nonumber\\
	&\leq\sum_{k=1}^{n}\sum_{i=1}^{k}\hat{\V}\left(Z_i\geq \frac{x}{k}\right)\nonumber\\
	&\leq\sum_{k=1}^{n}\sum_{i=1}^{k}\frac{k^2}{x^2} \mbE[Z_i^2]\nonumber\\
	&=\sum_{i=1}^{n}\sum_{k=i}^{n}\frac{k^2}{x^2}\mbE[Z_i^2]\nonumber\\
	&\leq\sum_{k=1}^{n}k^2\cdot x^{-2}\cdot \sum_{i=1}^{n}\mbE[Z_i^2]\nonumber\\
    &\leq \frac{(m+1)(m+2)(2m+3)}{6}\cdot x^{-2} \sum_{i=1}^{n}\mbE[Z_i^2].\nonumber
\end{align}
In the last inequality we use the truth that $\sum_{k=1}^{n}k^2=\frac{n(n+1)(2n+1)}{6}$ and notice that $n< m+1$. 
\par In the case of $n \geq m+1$, we take 
\begin{align}
	Z_{ki}\triangleq X_{i(m+1)+k}-\mbE[X_{i(m+1)+k}], i=0,\cdots ,[(n-k)/(m+1)].\nonumber
\end{align}
Then 
\begin{align}
	\hat{\V}\left(\max_{k\leq n}\sum_{i=1}^{k}(X_i-\mbE[X_i])\geq x\right)\nonumber
	&\leq\hat{\V}\left(\bigcup_{k=1}^{m+1}\max_{0\leq j\leq [(n-k)/(m+1)]}\sum_{i=0}^{j}Z_{ki}\geq\frac{x}{m+1}\right)\nonumber\\
	&\leq\sum_{k=1}^{m+1}\hat{\V}\left(\max_{0\leq j\leq [(n-k)/(m+1)]}\sum_{i=0}^{j}Z_{ki}\geq\frac{x}{m+1}\right)\nonumber\\
	&\leq C\sum_{k=1}^{m+1}\frac{(m+1)^2}{x^2}\sum_{i=0}^{[(n-k)/(m+1)]}\mbE[X_{i(m+1)+k}^2]\nonumber\\
	&=C(m+1)^2 \cdot x^{-2}\sum_{i=1}^{n}\mbE[X_i^2].\nonumber
\end{align}
In the second to last inequality we have used \eqref{eq2} because $\{Z_{ki},i=0,\cdots,[(n-k)/(m+1)]\}$ is a sequence of independent random variables.

The proof is completed.
\end{proof}
\par Next, we shall establish the Rademacher-Mensov inequality under the sublinear expectation, which will play a crucial role in the proof of the strong law of large numbers for orthogonal random variables. One can refer to \cite{Rév68} for the Rademacher-Mensov inequality in classical probability theory. 
\begin{lemma}\label{le3.4}
	Let $\{X_1,\cdots,X_n\}$ be a sequence of orthonormal random variables in the sublinear expectation space $(\Omega,\mathcal{H},\mbE)$, and $c_1,\cdots,c_n$ be a sequence of real numbers. Then
	\begin{equation}
		\mbE\left[\max_{1\leq k\leq n}\left(\sum_{j=1}^{k}c_jX_j\right)^2\right]\leq (\log_2{4n})^2\sum_{j=1}^{n}c_j^2.\label{eq5}
	\end{equation}
\end{lemma}
\begin{proof}
	We first prove for the case of $n=2^v$, $\forall v\in \mathbb{N}^*$. Let
	\begin{align}
		 \eta_j&\triangleq c_1X_1+\cdots+c_jX_j,\notag\\
			\psi_{\alpha\beta}&\triangleq c_{\alpha+1}X_{\alpha+1}+\cdots+c_{\beta+1}X_{\beta+1},\notag
	\end{align}
    where $\alpha\triangleq\mu\cdot 2^k$; $\beta\triangleq\beta(\alpha)\triangleq(\mu+1)\cdot 2^k$; $k=0,1,\cdots,v$; $\mu=0,1,\cdots,2^{v-k}-1.$ 
    \par We consider the $\eta_j$ as the sum of some $\psi_{\alpha\beta}$ and put
    $$\eta_j=\sum_i\psi_{\alpha_i\beta_i},$$
    where $\beta_1-\alpha_1>\beta_2-\alpha_2>\cdots$, then the number of the sum is less than $v$. 
    \par By the Cauchy inequality, for $j=1,\cdots,n$ we have
    \begin{align}
    	\eta^2_j&=(\sum_i\psi_{\alpha_i\beta_i}\cdot 1)^2\notag\\
    	        &\leq v\cdot\sum_i\psi^2_{\alpha_i\beta_i}\notag\\
    	        &\leq v\cdot\sum_{\alpha\beta}\psi^2_{\alpha\beta},\notag
    \end{align}
    where $\sum_{\alpha\beta}$ means that $\alpha$ and $\beta$ run through all their possible values. Therefore,
    \begin{equation}
    	\max_{1\leq j\leq n}\eta^2_j\leq v\cdot\sum_{\alpha\beta}\psi^2_{\alpha\beta}.\label{eq4}
    \end{equation} 
    Taking the sublinear on the both side of \eqref{eq4} and by the sub-additivity of $\mbE$, we have
    $$\mbE\left[\max_{1\leq j\leq N}\eta^2_j\right]\leq v\cdot\sum_{\alpha\beta}\mbE[\psi^2_{\alpha\beta}].$$
    By the orthogonality of $\{X_1,\cdots,X_n\}$, we have 
    $$\sum_{\alpha\beta}\mbE[\psi^2_{\alpha\beta}]\leq (v+1)\cdot \sum_{j=1}^{n}c^2_j.$$
Finally,
$$\mbE\left[\max_{1\leq j\leq N}\eta^2_j\right]\leq v\cdot (v+1)\cdot \sum_{j=1}^{n}c^2_j\leq (\log_{2}2n)^2\sum_{j=1}^{n}c^2_j,$$
and \eqref{eq5} holds for the case of $n=2^v$, $\forall v\in \mathbb{N}^*$. As for the case of $2^l<n<2^{l+1}$ for some $l\in \mathbb{N}^*$, we only need to take $c_{n+1}=0,\cdots,c_{2^{l+1}}=0$ in the case of $v=l+1$, and then the proof is completed.
\end{proof}

\par The next lemma is the convergence part of Borel-Cantelli lemma for a countable sub-additive capacity in Zhang \cite{Zha23}.

\begin{lemma}\label{le3.5}
Let $\{A_n,n\in \mathbb{N}^*\}$ be in $\mathcal{G}$, and $\hat{\V}^*$ be a countable sub-additive capacity on $\mathcal{G}$. If $\sum_{n=1}^{\infty}\hat{\V}^*(A_n)<\infty$, then
\begin{equation}
\hat{\V}^*(A_n,i.o.)=0,
\end{equation} 
 where $\{A_n,i.o.\}=\bigcap_{n=1}^\infty\bigcup_{i=n}^\infty A_i$.
\end{lemma}

\begin{lemma}
	\begin{itemize}
	\item[(i)] Let $\{a_n\}_{n\geq1}$ be a sequence with $a_n\nearrow \infty$, then for any $M>1$ there exists a sequence $\{n_k\}_{k\geq1}$ with $n_k\nearrow \infty$ such that
		\begin{equation}
			Ma_{n_k}\leq a_{n_{k+1}}\leq M^3a_{n_k+1}.
		\end{equation}
	\item[(ii)] Let $a_n\geq 0$, $n\in \mathbb{N}^*$ and $\sum_{n=1}^{\infty}a_n<\infty$, then there exists $b_n\geq 0$, $n\in \mathbb{N}^*$ such that $\sum_{n=1}^{\infty}b_n<\infty$ and $\frac{a_n}{b_n}\downarrow0$ as $n\rightarrow\infty$.
	\label{le3.6}
	\end{itemize}
	
\end{lemma}
$(i)$ is Lemma 3.3 of Wittmann \cite{Wittmann85}. For $(ii)$, we only need to take $b_n\triangleq\sqrt{\sum_{k=n}^{\infty}a_k}-\sqrt{\sum_{k=n+1}^{\infty}a_k}$.

\section{Main results}\label{sec4}
\par \ \ \ \ \ We often don't distinguish different sets with probability 1 in classical probability theory, but it is worth noting that we need to distinguish different sets with capacity 1. Because capacities generally does not have the property: $\forall A,B$ in $\mathcal{F}$, $\hat{\V}^*(A)=\hat{\V}^*(B)=1\Rightarrow \hat{\V}^*(A\cap B)=1$. For instance, let $\Omega=[0,2]$, $P_i$ is the probability on $\mathcal{B}([0,2])$ such that $P_i$ is the Lebesgue measure on $\mathcal{B}([i,i+1])$, $i=1,2$. $\hat{\V}(A)\triangleq\sup_{i=1,2}P_i(A)$, $\forall A \in \mathcal{B}([0,2])$. Then $\hat{\V}([0,1])=\hat{\V}([1,2])=1$, but $\hat{\V}([0,1]\cap [1,2])=\hat{\V}\{1\}=0$.  Hence, 
We need some techniques to overcome this crux when proving the law of large numbers for capacities.
\par We first show that the independence condition of Theorem 3.1 in Zhang \cite{GZ24} can be weakened to $m$-dependence.

\begin{theorem}\label{th4.1}
	Let $\{X_n;n\geq 1\}$ be a sequence of m-independent random variables in the sublinear expectation space $(\Omega,\mathcal{H},\mbE)$ and $\mbE[X_n]=\mbe[X_n]=0, n\geq 1$. Set $S_n=\sum_{i=1}^nX_i$ and suppose   $\{a_n\}_{n\geq1}$ is a sequence such that $1\leq a_n\nearrow\infty$ and
	$$\sum_{n=1}^{\infty}\frac{\mbE[X_n^2]}{a_n^2}<\infty.$$
   Then
   \begin{equation}
   		\hat{\V}^*\left(\limsup_{n\rightarrow\infty}\frac{S_n}{a_n}>0\enspace or\enspace\liminf_{n\rightarrow\infty}\frac{S_n}{a_n}<0\right)=0.\label{eq6}
   \end{equation}
	\end{theorem}
\begin{remark}
	By the sub-additivity of capacities, we know that for $A\in \mathcal{F}$,  $\hat{\V}^*(A)=0\Rightarrow\hat{\V}^*(A^c)=1$. Therefore, \eqref{eq6} means $\hat{\V}^*(lim_{n\rightarrow\infty}\frac{S_n}{a_n}=0)=1$.
	
\end{remark}
\par The next result shows a strong limit behavior of a sequence of blockwise $m$-dependent random variables with finite $r$-order Choquet integral under the sublinear expectation.
\begin{theorem}\label{th4.2}
	Let $\{n_i\}_{i\geq 1}$ be a given strictly increasing sequence of natural numbers, $\{X_n\}_{n\geq1}$ is $m$-dependent with respect to $\{n_i\}_{i\geq 1}$ in the sublinear expectation space $(\Omega,\mathcal{H},\mbE)$. Denote
	\par $I_k\triangleq\{i|[2^k,2^{k+1})\cap [n_i,n_{i+1})\neq \emptyset\}$, $v_k\triangleq \#I_k$, $k\in \mathbb{N}^*$,
	\par $[l_{ki},r_{ki})\triangleq[2^k,2^{k+1})\cap [n_i,n_{i+1})$, $i\in I_k,$
	\par $\Phi(n)\triangleq\max_{0\leq j\leq k}v_j$ if $n\in[2^k,2^{k+1})$, $S_n\triangleq\sum_{k=1}^{n}X_k.$
    \par Assume
	\begin{itemize}
	\item[(i)] There is a random variable $Z$ in $(\Omega,\mathcal{H},\mbE)$ such that $\frac{1}{n}\sum_{k=1}^{n}\hat{\V}(|X_k|>t)\leq C\hat{\V}(|Z|>t)$, $\forall n\in \mathbb{N}^*,\forall t>0,$
	where $C$ is a numarical constant,
	\item[(ii)] $C_{\hat{\V}}(|Z|^r)<\infty$, for some $r\in [1,2)$, 
	\end{itemize}
	then
	\begin{equation}
	\hat{\V}^*\left(\limsup_{n\rightarrow\infty}\frac{S_n-\sum_{k=1}^{n}\bE[X_k]}{n^{\frac{1}{r}}\Phi(n)}>0\ or\ \liminf_{n\rightarrow\infty}\frac{S_n-\sum_{k=1}^{n}\be[X_k]}{n^{\frac{1}{r}}\Phi(n)}<0\right)=0.\label{eq7}
	\end{equation}

	\begin{remark}
	 By $(i)$ and $(ii)$ we have $\frac{1}{n}\sum_{k=1}^{n}C_{\hat{\V}}(|X_k|^r)\leq C_{\hat{\V}}(|Z|^r)$. If $\bE[X_k]=\be[X_k]=0,k\geq 1$, then \eqref{eq7} means $\hat{\V}^*\left(\lim_{n\rightarrow\infty}\frac{S_n}{n^{\frac{1}{r}}\Phi(n)}=0\right)=1$. Morever, if $\{n_i\}_{i\geq 1}=\{2^i\}_{i\geq 1}$and $r=1$, then \eqref{eq7} means $\hat{\V}^*\left(\lim_{n\rightarrow\infty}\frac{S_n}{n}=0\right)=1$, which is a common strong law of large numbers.  $\Phi(n)$ can be estimated by taking some specific $\{n_i\}_{i\geq 1}$ and one can refer to \cite{Gap95} and \cite{Bo98} for details.
	 \end{remark}
	\begin{remark}
		Let $\{X_n\}_{n\geq1}$ be a sequence of i.i.d. random variables in the sublinear expectation space $(\Omega,\mathcal{H},\mbE)$ and $C_{\hat{\V}}(|X_1|^r)<\infty$ for some $r\in [1,2)$. We can take $\{n_i\}_{i\geq 1}=\{2^i\}_{i\geq 1}$ so that $\{X_n\}_{n\geq1}$ is $m$-dependent with respect to $\{2^i\}_{i\geq 1}$. Taking $Z=X_1$, it's easy to verify condition (i) and (ii) because $\{X_n\}_{n\geq1}$ is a sequence of identically distributed random variables. Then by \eqref{th4.2} we can get
			\begin{equation}
			\hat{\V}^*\left(\limsup_{n\rightarrow\infty}\frac{S_n-n\bE[X_1]}{n^{\frac{1}{r}}}>0\ or\ \liminf_{n\rightarrow\infty}\frac{S_n-n\be[X_1]}{n^{\frac{1}{r}}}<0\right)=0,\nonumber
		\end{equation} 
		which is (3.9) of theorem 3.4 in \cite{Zha23}.
	\end{remark}
    \end{theorem}
\par The next result is a strong law of large numbers for orthogonal sequences of random variables.
\begin{theorem}\label{th4.3}
	Let $\{X_n\}_{n\geq 1}$ in $(\Omega,\mathcal{H},\mbE)$ be orthogonal and satisfy $\mbE[X_n]=0$. Set $\sigma^2_n\triangleq\mbE[X^2_n]$, $S_n\triangleq\sum_{k=1}^{n}X_k$. If
	$$\sum_{k=1}^{\infty}\frac{\sigma^2_k}{k^2}(\log_2k)^2<\infty,$$
	then
	\begin{equation}
		\hat{\V}^*\left(\lim_{n\rightarrow\infty}\frac{S_n}{n}=0\right)=1.\label{eq8}
	\end{equation}
	
\end{theorem}

\par Actually, the requirement of orthogonality in Theorem \ref{th4.3} can be weakened to quasi-orthogonality.
\begin{corollary}\label{co4.1}
    Let $\{X_n\}_{n\geq 1}$ in $(\Omega,\mathcal{H},\mbE)$ be quasi-orthogonal instead of orthogonal. Other conditions are the same as in Theorem \ref{th4.3}, then \eqref{eq8} still holds.
\end{corollary}

\section{Proofs}
In this section, we shall give the proofs of the results in Section \ref{sec4}.
\begin{proof}[\bf{Proof of Theorem \ref{th4.1}}]
	By $(ii)$ of Lemma \ref{le3.6}, there exists a sequence $\epsilon_k\searrow0$ such that
$$
	\sum_{n=1}^\infty \frac{\mbE[X_n^2]}{\epsilon_n^2a_n^2}<\infty.
$$
By $(i)$ of Lemma \ref{le3.6}, for all $M>1$, there exists a sequence $n_k\nearrow\infty$ such that $$M a_{n_k}\leq a_{n_{k+1}}\leq M^3 a_{n_k+1}.$$
\par For $n_k+1\leq n\leq n_{k+1}$, $$S_n-S_{n_{k}}=\sum_{j=1}^{m+1}\sum_{l=0}^{L(n,j)}X_{n_k+j+l(m+1)},$$ where $L(n,j)\triangleq [(n-n_k-j)/(m+1)]$.
\par We know that $\{X_{n_k+j+i(m+1)}|i=0,\cdots,L(n,j)\}$ is a sequence of independent and centered random variables for given $j\in \{1,\cdots,m+1\}$ by the $m$-dependence of $\{X_n;n\geq 1\}$. And it's also clear that
$$\max_{n_k+1\leq n\leq n_{k+1}}(S_n-S_{n_k})\leq\sum_{j=1}^{m+1}\max_{0\leq L\leq L(n_{k+1},j)}\sum_{l=0}^{L}X_{n_k+j+l(m+1)}.$$ Hence,
$$\left\{\max_{n_k+1\leq n\leq n_{k+1}}(S_n-S_{n_k})\geq \epsilon_{n_k}a_{n_{k+1}}\right\}\subseteq\bigcup_{j=1}^{m+1}\left\{\max_{0\leq L\leq L(n_{k+1},j)}\sum_{l=0}^{L}X_{n_k+j+l(m+1)}\geq \frac{\epsilon_{n_k}a_{n_{k+1}}}{m+1}\right\}.$$
By \eqref{eq2}, we have
$$\hat{\V}\left(\max_{0\leq L\leq L(n_{k+1},j)}\sum_{l=0}^{L}X_{n_k+j+l(m+1)}\geq \frac{\epsilon_{n_k}a_{n_{k+1}}}{m+1}\right)\leq\frac{(m+1)^2}{\epsilon^2_{n_k}a^2_{n_k+1}}\sum_{l=0}^{L(n_{k+1},j)}\mbE[X^2_{n_k+j+l(m+1)}].$$
Then by the sub-additive of $\hat{\V}$, we have

\begin{align}
	\hat{\V}\left(\max_{n_k+1\leq n\leq n_{k+1}}(S_n-S_{n_k})\geq \epsilon_{n_k}a_{n_{k+1}}\right)
	&\leq\sum_{j=1}^{m+1}\hat{\V}\left(\max_{0\leq L\leq L(n_{k+1},j)}\sum_{l=0}^{L}X_{n_k+j+l(m+1)}\geq \frac{\epsilon_{n_k}a_{n_{k+1}}}{m+1}\right)\notag\\
	&\leq\frac{(m+1)^2}{\epsilon^2_{n_k}a^2_{n_{k+1}}}\sum_{j=1}^{m+1}\sum_{l=0}^{L(n_{k+1},j)}\mbE[X^2_{n_k+j+l(m+1)}]\notag\\
	&=\frac{(m+1)^2}{\epsilon^2_{n_k}a^2_{n_{k+1}}}\sum_{j=n_k+1}^{n_{k+1}}\mbE[X^2_j]\notag\\
    &\leq(m+1)^2\sum_{j=n_k+1}^{n_{k+1}}\frac{\mbE[X^2_j]}{\epsilon^2_ja^2_j}.\notag
\end{align}
So we still have $$\sum_{k=1}^{\infty}\hat{\V}\left(\max_{n_k+1\leq n\leq n_{k+1}}(S_n-S_{n_k})\geq \epsilon_{n_k}a_{n_{k+1}}\right)<\infty,$$
and the rest of the proof is the same as the proof of Theorem 3.1 in \cite{GZ24}.
\end{proof}
\begin{proof}[\bf{Proof of Theorem \ref{th4.2}}]
	Let $\overline{X}_n\triangleq(-n^{\frac{1}{r}})\vee X_n\wedge(n^{\frac{1}{r}})$, $\overline{Z}_n\triangleq(-n^{\frac{1}{r}})\vee Z\wedge(n^{\frac{1}{r}}),n\in \mathbb{N}^*$. We show the proof in four steps.
	\par \textbf{Step 1.} We show for the truncated random variables $\{\overline{X}_n\}_{n\geq1}$,
	\begin{equation}
		\hat{\V}^*\left(\frac{1}{2^{k/r}\Phi(2^k)}\max_{2^k\leq n<2^{k+1}}\sum_{j=2^k}^{n}\overline{X}_j-\bE[\overline{X}_j]>\epsilon,i.o.\right)=0, \forall \epsilon>0.\label{eq9}
	\end{equation}
	\par Set $Y_j\triangleq \overline{X}_j-\bE[\overline{X}_j]$. For any $\epsilon>0$,
	
		$$\sum_{k=1}^{\infty}\hat{\V}\left(\frac{1}{2^{k/r}\Phi(2^k)}\max_{2^k\leq n<2^{k+1}}\sum_{j=2^k}^{n}Y_j>\epsilon\right)$$
	\begin{align}
		&=\sum_{k=1}^{\infty}\hat{\V}\left(\max_{2^k\leq n<2^{k+1}}\sum_{j=2^k}^{n}Y_j>\epsilon\cdot 2^{k/r}\Phi(2^k) \right)\notag\\
		&\leq\sum_{k=1}^{\infty}\hat{\V}\left(\bigcup_{i\in I_k}\left\{\max_{l_{ki}\leq n<r_{ki}}\sum_{j=l_{ki}}^{n}Y_j>\epsilon\cdot \frac{2^{k/r}\Phi(2^k)}{v_k}\right\}\right)\notag\\
		&\leq\sum_{k=1}^{\infty}\sum_{i\in I_k}\hat{\V}\left(\max_{l_{ki}\leq n<r_{ki}}\sum_{j=l_{ki}}^{n}Y_j>\epsilon\cdot \frac{2^{k/r}\Phi(2^k)}{v_k}\right)\notag\\
		&\leq
		\sum_{k=1}^{\infty}\sum_{i\in I_k}C\cdot\frac{v_k^2}{\epsilon^{2}2^{2k/r}\Phi^2(2^k)}\sum_{j=l_{ki}}^{r_{ki}-1}\mbE[\overline{X}^2_j]\notag\\
		&\leq
		C\sum_{k=1}^{\infty}\sum_{i\in I_k}\frac{1}{2^{2k/r}}\sum_{j=l_{ki}}^{r_{ki}-1}\mbE[\overline{X}^2_j]\notag\\
		&= C\cdot \sum_{k=1}^{\infty}\frac{1}{2^{2k/r}}\sum_{j=2^k}^{2^{k+1}-1}\mbE[\overline{X}^2_j].\notag
	\end{align}
The third inequality follows from the sub-additivity of $\hat{\V}$ and the fourth inequality follows from Lemma \ref{le3.3}. The second to last inequality is due to $v^2_k/\Phi^2(2^k)\leq1$. By \eqref{eq3} of Lemma \ref{le3.1}, we have
\begin{align}
	\sum_{j=2^k}^{2^{k+1}-1}\mbE[\overline{X}^2_j]
	&\leq \sum_{j=2^k}^{2^{k+1}-1}C_{\hat{\V}}[\overline{X}^2_j]\notag\\
	&=\sum_{j=2^k}^{2^{k+1}-1}\int_{0}^{+\infty}\hat{\V}(|\overline{X}_j|^2\geq t)dt\notag\\
    &\leq\sum_{j=1}^{2^{k+1}}\int_{0}^{(2^{\frac{k+1}{r}})^2}\hat{\V}(|\overline{X}_j|^2\geq t)dt\notag\\
	&= \sum_{j=1}^{2^{k+1}}\int_{0}^{(2^{\frac{k+1}{r}})^2}\hat{\V}(|X_j|> \sqrt{t})dt\notag\\
	&\leq \int_{0}^{(2^{\frac{k+1}{r}})^2}2^{k+1}\hat{\V}(|Z|> \sqrt{t})dt\notag\\
	&=2^{k+1}C_{\hat{\V}}[|\overline{Z}_{2^{k+1}}|^2].\notag
\end{align}
	Then
	$$\sum_{k=1}^{\infty}\frac{1}{2^{2k/r}}\sum_{j=2^k}^{2^{k+1}-1}\mbE[\overline{X}^2_j]$$
	\begin{align}
		&\leq C\sum_{k=0}^{\infty}\frac{1}{2^{2k/r}}2^{k}C_{\hat{\V}}[|\overline{Z}_{2^{k}}|^2]\notag\\
		&=C\sum_{k=0}^{\infty}2^k\int_{0}^{+\infty}\hat{\V}\left(\left|\frac{\overline{Z}_{2^k}}{2^{k/r}}\right|^2\geq t\right)dt\notag\\
		&=C\sum_{k=0}^{\infty}2^k\int_{0}^{1}\hat{\V}\left(\left|\frac{\overline{Z}_{2^k}}{2^{k/r}}\right|^2\geq t\right)dt\notag\\
		&\leq C\sum_{k=0}^{\infty}2^k\int_{0}^{1}t\cdot\hat{\V}\left(\left|\frac{\overline{Z}_{2^k}}{2^{k/r}}\right|\geq t\right)dt\notag\\
		&=C\int_{0}^{1}t\cdot\sum_{k=0}^{\infty}2^k\hat{\V}(|Z|\geq2^{k/r}t)dt\notag\\
		&=C\int_{0}^{1}t\cdot\sum_{k=0}^{\infty}2^k\hat{\V}\left(\left|\frac{Z}{t}\right|^r\geq2^{k}\right)dt\notag\\
		&\leq C\int_{0}^{1}t\cdot C_{\hat{\V}}\left[\left|\frac{Z}{t}\right|^r \right]dt\notag\\
		&=C\cdot C_{\hat{\V}}[|Z|^r]\cdot \int_{0}^{1}\frac{1}{t^{r-1}}dt<\infty\notag.
	\end{align}
Therefore, 
\begin{equation}
	\sum_{k=1}^{\infty}\hat{\V}^*\left(\frac{1}{2^{k/r}\Phi(2^k)}\max_{2^k\leq n<2^{k+1}}\sum_{j=2^k}^{n}Y_j>\epsilon\right)\leq\sum_{k=1}^{\infty}\hat{\V}\left(\frac{1}{2^{k/r}\Phi(2^k)}\max_{2^k\leq n<2^{k+1}}\sum_{j=2^k}^{n}Y_j>\epsilon\right)<\infty, \forall\epsilon>0.\label{eq10}
\end{equation} By Borel-Cantelli lemma \ref{le3.5}, \eqref{eq10} yields \eqref{eq9}. 
\par \textbf{Step 2.} We show 
\begin{equation}
\hat{\V}^*(X_k\neq\overline{X}_k,i.o.)=0.\label{eq11}
\end{equation}
\par By Lemma \ref{le3.5}, we only need to prove $\sum_{k=1}^{\infty}\hat{\V}(X_k\neq\overline{X}_k)<\infty$. 
$$\sum_{k=1}^{\infty}\hat{\V}(X_k\neq\overline{X}_k)$$
\begin{align}
	&=\sum_{k=1}^{\infty}\hat{\V}(|X_k|>k^{\frac{1}{r}})\notag\\
	&\leq\sum_{k=0}^{\infty}\sum_{n=2^k}^{2^{k+1}-1}\hat{\V}(|X_n|>2^{k/r})\notag\\
	&\leq C\sum_{k=0}^{\infty}2^k\hat{\V}(|Z|>2^{k/r})\notag\\
	&\leq C\cdot C_{\hat{\V}}[|Z|^r]<\infty.\notag
\end{align}
Hence, \eqref{eq11} holds.
\par \textbf{Step 3.} We show
\begin{equation}
	\lim_{n\rightarrow\infty}\frac{\sum_{k=1}^{n}\bE[X_k]-\bE[\overline{X}_k]}{n^{\frac{1}{r}}}=0.\label{eq12}
\end{equation}
\par For $1\leq r<2,$
\begin{align}
	\left|\sum_{k=1}^{n}\frac{\bE[X_k]-\bE[\overline{X}_k]}{k^\frac{1}{r}}\right|
    &\leq\sum_{k=1}^{n}\frac{\bE[(|X_k|-k^{\frac{1}{r}})^+]}{k^{\frac{1}{r}}}\notag\\
    &\leq\sum_{k=1}^{\infty}\frac{C_{\hat{\V}}[(|X_k|-k^{\frac{1}{r}})^+]}{k^{\frac{1}{r}}}\notag\\
    &\leq\sum_{k=0}^{\infty}\sum_{n=2^k}^{2^{k+1}-1}\frac{C_{\hat{\V}}[(|X_n|-2^{\frac{k}{r}})^+]}{2^{\frac{k}{r}}}\notag\\
    &=\sum_{k=0}^{\infty}\sum_{n=2^k}^{2^{k+1}-1}\frac{\int_{2^{k/r}}^{+\infty}\hat{\V}(|X_n|>t)dt}{2^{\frac{k}{r}}}\notag\\ 
    &\leq C\sum_{k=0}^{\infty}2^k\cdot\frac{\int_{2^{k/r}}^{+\infty}\hat{\V}(|Z|>t)dt}{2^{\frac{k}{r}}}\notag\\ 
    &=C\sum_{k=0}^{\infty}2^k\cdot\int_{1}^{+\infty}\hat{\V}(|Z|>2^{k/r}m)dm\notag\\
    &=C\int_{1}^{+\infty}\sum_{k=0}^{\infty}2^k\cdot\hat{\V}\left(\left|\frac{Z}{m}\right|^r>2^k\right)dm\notag\\
    &\leq C\int_{1}^{+\infty}C_{\hat{\V}}\left(\left|\frac{Z}{m}\right|^r\right)dm\notag\\
    &=C\cdot C_{\hat{\V}}[|Z|^r]\cdot \int_{1}^{+\infty}\frac{1}{m^r}\ dm<\infty\notag.
\end{align}

Therefore, according to the famous Kronecker lemma, \eqref{eq12} holds.
\par \textbf{Step 4.} We show
\begin{equation}
	\hat{\V}^*\left(\limsup_{n\rightarrow\infty}\frac{S_n-\sum_{k=1}^{n}\bE[X_k]}{n^{1/r}\Phi(n)}>\epsilon\right)=0, \forall\epsilon>0.\label{eq13}
\end{equation}
\par Let $2^k\leq n<2^{k+1}$, $\epsilon>0$. For every $\omega\in \left\{\frac{1}{2^{k/r}\Phi(2^k)}\max_{2^k\leq n<2^{k+1}}\sum_{j=2^k}^{n}\overline{X}_j-\bE[\overline{X}_j]>\epsilon,i.o.\right\}^c,$ there exists a positive integer $k_0(\omega)$ such that for every $k\geq 2^{k_0(\omega)}$, 
\begin{equation}
	\frac{1}{2^{k/r}\Phi(2^k)}\max_{2^k\leq n<2^{k+1}}\sum_{j=2^k}^{n}Y_j\leq\epsilon,\label{eq15}
\end{equation}
where $Y_j\triangleq \overline{X}_j-\bE[\overline{X}_j].$
\par On $\left\{\frac{1}{2^{k/r}\Phi(2^k)}\max_{2^k\leq n<2^{k+1}}\sum_{j=2^k}^{n}\overline{X}_j-\bE[\overline{X}_j]>\epsilon,i.o.\right\}^c$, we have
\begin{align}
	\frac{\sum_{j=1}^{n}Y_j}{n^{1/r}\Phi(n)}
	&\leq \frac{\sum_{j=1}^{2^k-1}Y_j}{n^{1/r}\Phi(n)}+\frac{\max_{2^k\leq i<2^{k+1}}\sum_{j=2^k}^{i}Y_j}{n^{1/r}\Phi(n)}\notag\\
	&=\frac{\sum_{j=1}^{2^{k_0}-1}Y_j}{n^{1/r}\Phi(n)}+\frac{\sum_{m=k_0}^{k-1}\sum_{j=2^m}^{2^{m+1}-1}Y_j}{n^{1/r}\Phi(n)}+\frac{2^{k/r}\Phi(2^k)}{n^{1/r}\Phi(n)}\cdot \frac{\max_{2^k\leq i<2^{k+1}}\sum_{j=2^k}^{i}Y_j}{2^{k/r}\Phi(2^k)}.\label{eq14}
\end{align}
The first term in \eqref{eq14} converges to zero when $n\rightarrow\infty$. By $\frac{2^{k/r}\Phi(2^k)}{n^{1/r}\Phi(n)}\leq 1$ and \eqref{eq15}, we have $\limsup_{n\rightarrow\infty}\frac{2^{k/r}\Phi(2^k)}{n^{1/r}\Phi(n)}\cdot \frac{\max_{2^k\leq i<2^{k+1}}\sum_{j=2^k}^{i}Y_j}{2^{k/r}\Phi(2^k)}\leq \epsilon$ for the third term in \eqref{eq14}. As for the middle term,
\begin{align}
	\frac{\sum_{m=k_0}^{k-1}\sum_{j=2^m}^{2^{m+1}-1}Y_j}{n^{1/r}\Phi(n)}
	&\leq\sum_{m=k_0}^{k-1}\frac{2^{m/r}\Phi(2^m)}{n^{1/r}\Phi(n)}\cdot\frac{\max_{2^m\leq n<2^{m+1}}\sum_{j=2^m}^{n}Y_j}{2^{m/r}\Phi(2^m)}\notag\\
	&\leq\epsilon\cdot \sum_{m=k_0}^{k-1}\frac{2^{m/r}\Phi(2^m)}{n^{1/r}\Phi(n)}\notag\\
	&\leq\epsilon\cdot\frac{\sum_{m=k_0}^{k-1}2^{m/r}}{n^{1/r}}\notag\\
	&\leq\epsilon\cdot\frac{\frac{1}{ln(2^{1/r})}\cdot2^{k/r}}{n^{1/r}}\notag\\
	&\leq 3\epsilon,
\end{align}
where the fourth inequality follows from the truth that for $\alpha>1, \sum_{m=0}^{k-1}\alpha^m\leq \frac{\alpha^k}{ln\alpha}.$ Hence, we have $\limsup_{n\rightarrow\infty}\frac{\sum_{k=1}^{n}\overline{X}_k-\bE[\overline{X}_k]}{n^{1/r}\Phi(n)}\leq4\epsilon$ on $\left\{\frac{1}{2^{k/r}\Phi(2^k)}\max_{2^k\leq n<2^{k+1}}\sum_{j=2^k}^{n}Y_j>\epsilon,i.o.\right\}^c$, which implies
\begin{equation}
	\hat{\V}^*\left(\limsup_{n\rightarrow\infty}\frac{\sum_{k=1}^{n}\overline{X}_k-\bE[\overline{X}_k]}{n^{1/r}\Phi(n)}>4\epsilon\right)\leq\hat{\V}^*\left(\frac{1}{2^{k/r}\Phi(2^k)}\max_{2^k\leq n<2^{k+1}}\sum_{j=2^k}^{n}Y_j>\epsilon,i.o.\right)=0, \forall \epsilon>0.\label{eq16}
\end{equation} 

\par It's clear that
$$\frac{S_n-\sum_{k=1}^{n}\bE[X_k]}{n^{1/r}\Phi(n)}=\frac{\sum_{k=1}^{n}X_k-\overline{X}_k}{n^{1/r}\Phi(n)}+
\frac{\sum_{k=1}^{n}\overline{X}_k-\bE[\overline{X}_k]}{n^{1/r}\Phi(n)}+
\frac{\sum_{k=1}^{n}\bE[\overline{X}_k]-\bE[X_k]}{n^{1/r}\Phi(n)},$$
which implies 
$$\left\{\limsup_{n\rightarrow\infty}\frac{S_n-\sum_{k=1}^{n}\bE[X_k]}{n^{1/r}\Phi(n)}>\epsilon\right\}\subseteq\{X_k\neq\overline{X}_k,i.o.\}\bigcup\left\{\limsup_{n\rightarrow\infty}\frac{\sum_{k=1}^{n}\overline{X}_k-\bE[\overline{X}_k]}{n^{1/r}\Phi(n)}>\epsilon\right\},\forall\epsilon>0.$$
Therefore, by the sub-additivity of $\hat{\V}^*$, \eqref{eq11} and \eqref{eq16} we have \eqref{eq13}.
\par So far, we have obtained \eqref{eq9}, \eqref{eq11}, \eqref{eq12}, \eqref{eq13}. 
\par We notice that $$\hat{\V}^*\left(\limsup_{n\rightarrow\infty}\frac{S_n-\sum_{k=1}^{n}\bE[X_k]}{n^{1/r}\Phi(n)}>0\right)=\hat{\V}^*\left(\bigcup_{m=1}^{\infty}\left\{\limsup_{n\rightarrow\infty}\frac{S_n-\sum_{k=1}^{n}\bE[X_k]}{n^{1/r}\Phi(n)}>\frac{1}{m}\right\}\right).$$
Hence, by the countable sub-additivity of $\hat{\V}^*$ and \eqref{eq13} we have $$\hat{\V}^*\left(\limsup_{n\rightarrow\infty}\frac{S_n-\sum_{k=1}^{n}\bE[X_k]}{n^{1/r}\Phi(n)}>0\right)=0.$$
Taking $X_k$ by $-X_k$, we have
$$\hat{\V}^*\left(\liminf_{n\rightarrow\infty}\frac{S_n-\sum_{k=1}^{n}\be[X_k]}{n^{1/r}\Phi(n)}<0\right)=0.$$
Finally, by the sub-additivity of $\hat{\V}^*$, \eqref{eq7} is proved.
\end{proof}
Next, we show the proof of Theorem \ref{th4.3}.
\begin{proof}[\bf{Proof of Theorem \ref{th4.3}}]
	\par 	By $(ii)$ of Lemma \ref{le3.6}, there exists a sequence $\epsilon_k\searrow0$ such that
	$$
	\sum_{k=1}^{\infty} \frac{\sigma^2_k}{\epsilon_k^2k^2}(\log_2k)^2<\infty.
	$$
	Then
	\begin{align}
	\hat{\V}^*\left(\max_{2^k+1\leq n\leq2^{k+1}}\left|S_n-S_{2^k}\right|\geq\epsilon_{2^k}\cdot 2^{k+1} \right)
	&\leq\frac{\mbE[\left(\max_{2^k+1\leq n\leq2^{k+1}}|S_n-S_{2^k}|\right)^2]}{\epsilon^2_{2^k}(2^{k+1})^2}\notag\\
	&\leq (\log_22^{k+2})^2\frac{\sum_{j=2^k+1}^{2^{k+1}}\mbE[X^2_j]}{\epsilon^2_{2^k}(2^{k+1})^2}\notag\\
	&\leq(\log_22^{k+2})^2\sum_{j=2^k+1}^{2^{k+1}}\frac{\mbE[X^2_j]}{\epsilon^2_{j}j^2}\notag\\
	&\leq\sum_{j=2^k+1}^{2^{k+1}}\frac{\mbE[X^2_j]}{\epsilon^2_{j}j^2}(2+\log_2j)^2,\notag
	\end{align}
where we use the Rademacher-Mensov inequality \eqref{eq5} in the second inequality. Hence, $$\sum_{k=1}^{\infty}\hat{\V}^*\left(\max_{2^k+1\leq n\leq2^{k+1}}\left|S_n-S_{2^k}\right|\geq\epsilon_{2^k}\cdot 2^{k+1} \right)\leq\sum_{j=1}^{\infty}\frac{\sigma^2_j}{\epsilon_j^2j^2}(2+\log_2j)^2<\infty,$$
	and by lemma \ref{le3.5} we have $\hat{\V}^*\left(\max_{2^k+1\leq n\leq2^{k+1}}\left|S_n-S_{2^k}\right|\geq\epsilon_{2^k}\cdot 2^{k+1},i.o.\right)=0.$
	\par For every $\omega\in\left\{\max_{2^k+1\leq n\leq2^{k+1}}\left|S_n-S_{2^k}\right|\geq\epsilon_{2^k}\cdot 2^{k+1},i.o.\right\}^c$, there exists a positive integer $K_0(\omega)$ such that for every $k\geq K_0(\omega)$, $\max_{2^k+1\leq n\leq2^{k+1}}\left|S_n-S_{2^k}\right|<\epsilon_{2^k}\cdot 2^{k+1}$. Let $n>2^{K_0(\omega)+1}$, $2^k+1\leq n\leq 2^{k+1}$, $K_0(\omega)\leq k_0< k$. On $\left\{\max_{2^k+1\leq n\leq2^{k+1}}\left|S_n-S_{2^k}\right|\geq\epsilon_{2^k}\cdot 2^{k+1},i.o.\right\}^c$, we have
	\begin{align}
		|S_n|
		&\leq|S_n-S_{2^k}|+\sum_{i=k_0+1}^{k}|S_{2^i}-S_{2^{i-1}}|+|S_{2^{k_0}}|\notag\\
		&\leq\epsilon_{2^k}\cdot2^{k+1}+\sum_{i=k_0+1}^{k}\epsilon_{2^{i-1}}2^i+|S_{2^{k_0}}|\notag\\
		&\leq\epsilon_{2^{k_0}}\cdot2^{k+1}(1+\frac{1}{2}+\cdots+\frac{1}{2^{k-k_0}})+|S_{2^{k_0}}|\notag\\
		&\leq\epsilon_{2^{k_0}}\cdot2^{k+1}\cdot 2+|S_{2^{k_0}}|.\notag
	\end{align}
	Then
	\begin{align}
		\frac{|S_n|}{n}
		&\leq\epsilon_{2^{k_0}}\cdot\frac{2^{k+1}}{n}\cdot 2+\frac{|S_{2^{k_0}}|}{n}\notag\\
		&\leq\epsilon_{2^{k_0}}\cdot\frac{2^{k+1}}{2^k}\cdot 2+\frac{|S_{2^{k_0}}|}{n}\notag\\
		&=4\epsilon_{2^{k_0}}+\frac{|S_{2^{k_0}}|}{n}.\notag
	\end{align}
	So $$\limsup_{n\rightarrow\infty}\frac{|S_n|}{n}\leq4\epsilon_{2^{k_0}},$$\\
	then letting $k_0\rightarrow\infty$ we have $\lim_{n\rightarrow\infty}\frac{S_n}{n}=0$, which means $$\left\{\max_{2^k+1\leq n\leq2^{k+1}}\left|S_n-S_{2^k}\right|\geq\epsilon_{2^k}\cdot 2^{k+1},i.o.\right\}^c\subseteq\left\{\lim_{n\rightarrow\infty}\frac{S_n}{n}=0\right\}.$$ \\
	Therefore, 
	$$1=\hat{\V}^*\left(\left\{\max_{2^k+1\leq n\leq2^{k+1}}\left|S_n-S_{2^k}\right|\geq\epsilon_{2^k}\cdot 2^{k+1},i.o.\right\}^c\right)\leq\hat{\V}^*\left(\lim_{n\rightarrow\infty}\frac{S_n}{n}=0\right).$$
	\eqref{eq8} is proved.
	\end{proof}
	\begin{proof}[\bf{Proof of Corollary \ref{co4.1}}]
		We just need to show that Rademacher-Mensov type inequality is still valid for the quasi-orthogonal sequence of random variables $\{X_n\}_{n\geq1}$, and the rest of the proof is the same as that of Theorem \ref{th4.3}. By the definition of quasi-orthogonality (Definiton \ref{de1}) we know that there exists a nonnegative sequence $\{f(j):j=0,1,\cdots\}$ and $\sum_{j=0}^{\infty}f(j)<\infty$ such that $|\mbE[X_kX_l]|\leq\sigma_k\sigma_lf(|k-l|)$,$\forall k,l=1,2,\cdots.$ So,
		\begin{align}
			\mbE\left[\left(\sum_{k=1}^{n}X_k\right)^2\right]
			&=\mbE\left[\sum_{k=1}^{n}X^2_k+2\sum_{j=1}^{n-1}\sum_{k=1}^{n-j}X_kX_{k+j}\right]\notag\\
			&\leq\sum_{k=1}^{n}\mbE[X^2_k]+2\sum_{j=1}^{n-1}\sum_{k=1}^{n-j}\mbE[X_kX_{k+j}]\notag\\
			&\leq\sum_{k=1}^{n}\sigma^2_k+2\sum_{j=1}^{n-1}\sum_{k=1}^{n-j}\sigma_k\sigma_{k+j}f(j)\notag\\
			&\leq\sum_{k=1}^{n}\sigma^2_k+\sum_{j=1}^{n-1}f(j)\sum_{k=1}^{n-j}(\sigma^2_k+\sigma^2_{k+j})\notag\\
			&\leq\left(1+2\sum_{j=1}^{\infty}f(j)\right)\sum_{k=1}^{n}\sigma^2_k.\notag
		\end{align}
		Then manipulating the proof of Lemma \ref{le3.4} again, we have 
		$$\mbE[\psi^2_{\alpha\beta}]\leq\left(1+2\sum_{j=1}^{\infty}f(j)\right)\sum_{j=\alpha}^{\beta}c^2_j,$$
		and finally, we have $$\mbE\left[\max_{1\leq j\leq N}\eta^2_j\right]\leq \left(1+2\sum_{j=1}^{\infty}f(j)\right)(\log_{2}4n)^2\sum_{j=1}^{n}c^2_j.$$
		The proof is completed.
		\end{proof}

\bibliographystyle{plain}
%\bibliography{m-SLLNreference}

\end{document}